\documentclass[11pt]{amsart}
\usepackage[usenames,dvipsnames]{xcolor}
\usepackage{geometry}
\usepackage{xspace}

\usepackage[numbers]{natbib}

\usepackage{amsmath,amssymb,amsthm}
\usepackage{enumerate}
\usepackage{aliascnt}
\usepackage{hyperref}
	\hypersetup{colorlinks=true, citecolor=blue, linkcolor=BrickRed, urlcolor=OliveGreen, pdfstartview={FitH}}
	
	\newtheorem{theorem}{Theorem}[section]

	\newaliascnt{lemma}{theorem}
	\newtheorem{lemma}[lemma]{Lemma}
	\aliascntresetthe{lemma}

	\newaliascnt{corollary}{theorem}
	
	\aliascntresetthe{corollary}

	\newaliascnt{proposition}{theorem}
	
	\aliascntresetthe{proposition}

	\theoremstyle{definition}
	\newaliascnt{definition}{theorem}
	
	\aliascntresetthe{definition}

	\newaliascnt{remark}{theorem}
	\newtheorem{remark}[remark]{Remark}
	\aliascntresetthe{remark}

	\newaliascnt{example}{theorem}
	
	\aliascntresetthe{example}

	\usepackage[capitalize,noabbrev]{cleveref}
	\crefname{theorem}{Theorem}{Theorems}
	\Crefname{theorem}{Theorem}{Theorems}
	\crefname{lemma}{Lemma}{Lemmas}
	\Crefname{lemma}{Lemma}{Lemmas}
	\crefname{corollary}{Corollary}{Corollaries}
	\Crefname{corollary}{Corollary}{Corollaries}
	\crefname{proposition}{Proposition}{Propositions}
	\Crefname{proposition}{Proposition}{Propositions}
	\crefname{definition}{Definition}{Definitions}
	\Crefname{definition}{Definition}{Definitions}
	\crefname{remark}{Remark}{Remarks}
	\Crefname{remark}{Remark}{Remarks}
	\crefname{example}{Example}{Examples}
	\Crefname{example}{Example}{Examples}

\title{Uniform Bounds for Digit-Appending Fibonacci Walks}

\author[Scott Duke Kominers]{Scott Duke Kominers}
\address{Harvard Business School; Department of Economics and Center of Mathematical Sciences and Applications, Harvard University; and a16z crypto}
\email{kominers@fas.harvard.edu}
\thanks{I gratefully acknowledge helpful conversations with Ben Golub, Steven~J.~Miller, Ken Ono, and Jesse Shapiro. Additionally, I used LLMs to assist with some of the computations and coding in the preparation of this article, particularly GPT-5-Pro and Claude-Sonnet-4.5 (both accessed via Poe with the support of Quora, where I am an advisor). The problem, methods, and eventual written form are my own; and of course any errors remain my responsibility.}


\begin{document}

\begin{abstract}
Building on the work of \citet{miller2022walking}, we show that it is impossible to ``walk to infinity'' along the Fibonacci sequence in any integer base $b\geq 2$ when at most $N$ digits are appended per step. Our proof method is base-independent, yielding the bound
\[
L \;\leq\; 2N\log_\varphi b \,+\, O(1),
\]
uniformly in the starting term, without relying on base-specific periodicity computations (here, $\varphi=\frac{1+\sqrt{5}}{2}$). Our approach extends to certain Lucas sequences.
\end{abstract}

\maketitle

\section{Introduction}

An \emph{at-most-$N$-digit-appending walk} along the Fibonacci sequence in base $b\geq 2$ proceeds as follows: start from a Fibonacci number $M_0:=F_m$ and, at each step $M_0, M_1, M_2, \ldots$, append between $1$ and $N$ base-$b$ digits to the right (leading zeros in the appended block are allowed), producing 
\[
M_{k+1} \;:=\; b^t \, M_k \,+\, r, \qquad (1\leq t\leq N;\; 0\leq r < b^t),
\]
with the requirement that each $M_k$ be a Fibonacci number. The central question we consider is whether such a walk can continue indefinitely, i.e., \textit{is it possible to walk to infinity along the Fibonacci sequence in base $b$}?

In base $10$, \citet*{miller2022walking} proved that every at-most-$N$-digit-appending walk terminates and obtained sharp bounds on the maximum-possible walk length using base-$10$--specific modular information---specifically, the base-$10$ Pisano period.\footnote{See also~\cite{miller2024walking}, in which \citeauthor*{miller2024walking} examined the walking-to-infinity question for a variety of other number-theoretic sequences, including primes, square-free numbers, and perfect squares.} 

In this paper we give a unified, base-independent proof of termination that avoids any base-specific periodicity. Using only standard Fibonacci identities and a simple rigidity argument, we obtain an explicit bound that is linear in $N$: The maximum number of steps $L$ in a base-$b$ at-most-$N$-digit-appending walk along the Fibonacci sequence is 
\begin{equation}
L \;\leq\; 2N\log_\varphi b \,+\, O(1);
\label{eq:linear-bound}
\end{equation}
in particular, any such walk is finite.

Making the argument base-independent entails a nontrivial trade-off in terms of efficiency: for base $10$, our bound \eqref{eq:linear-bound} is weaker than the $O(\log N)$ bound obtained by~\citet{miller2022walking}, but our method applies verbatim to any base $b\geq 2$ and extends qualitatively to Lucas sequences.

\section{Preliminaries}

We write $(F_n)$ for the \emph{Fibonacci sequence}, defined by $F_0=0$, $F_1=1$, and $F_{n+1}=F_n+F_{n-1}$. We denote $\varphi=\frac{1+\sqrt{5}}{2}$.

We begin by recalling two Fibonacci identities that were core to the argument of \cite{miller2022walking}.

\begin{lemma}[{\cite[Lemma~2.1]{miller2022walking}}]\label{lem:Miller2.1}
For all integers $m\geq 1$ and $k\geq 1$,
\begin{equation}
F_{k+1}\,F_m \;\leq\; F_{m+k} \;\leq\; F_{k+2}\,F_m.
\label{eq:F(k+1)F(m)}
\end{equation}
\end{lemma}

\begin{lemma}[{\cite[Lemma~2.2]{miller2022walking}}]\label{lem:Miller2.2}
For all integers $m\geq k\geq 2$,
\[
F_{m+k} \;=\; \big(F_{k+2}-F_{k-2}\big)F_m \,+\, (-1)^{k+1} F_{m-k}.
\]
\end{lemma}

\Cref{lem:Miller2.1} shows that the Fibonacci number at index $m+k$ grows roughly like the product of the $m$-th and $k$-th Fibonacci numbers. Meanwhile, \cref{lem:Miller2.2} shows how to ``jump ahead'' by $k$ positions in the Fibonacci sequence: the answer is almost a multiple of the starting value $F_m$, with a correction term from $k$ steps back.

We also use the elementary growth bound
\begin{equation}\label{eq:phi-bounds}
\varphi^{m-2} \;\leq\; F_m \;\leq\; \varphi^{m-1}\qquad(m\geq 1),
\end{equation}
which follows from Binet's formula.

\section{Bounding Base-$b$ Fibonacci Walks}

We first state our main result (\cref{thm:main}), then derive two intermediate lemmata (\cref{lem:jump-bound,lem:rigidity}), and then finally give the proof of the main theorem.

\begin{theorem}\label{thm:main}
We fix $b\geq 2$ and $N\geq 1$, and define
\[
K \;:=\; \Big\lceil 1 + \log_\varphi(2b^N)\Big\rceil,
\qquad
n_* \;:=\; \max\{n\geq 0:\; F_n \leq b^N-1\}.
\]
No base-$b$, $N$-digit-appending step to a Fibonacci number is possible from any $F_m$ with $m\geq n_*+K+1$.\footnote{Note that $n_* \geq 2$ since $F_2 = 1 \leq b^N - 1$ for $N \geq 1, b \geq 2$.} Consequently, any at-most-$N$-digit-appending walk along Fibonacci numbers in base-$b$ can have at most
\[
L \;\leq\; n_* + K \;\leq\; 2N\log_\varphi b \,+\, \log_\varphi 2 \,+\, 4
\]
steps. In particular, $L=O(N\log_\varphi b)$ uniformly in the starting term.
\end{theorem}

The proof of \cref{thm:main} relies on two results that tightly constrain the circumstances under which we can reach a Fibonacci number from another Fibonacci number by appending at most $N$ base-$b$ digits. \Cref{lem:jump-bound}, building on \cref{lem:Miller2.2}, shows that such a step cannot jump ``too far ahead'' in the sequence: the jump size $k$ is at most logarithmic in $b^N$. \Cref{lem:rigidity}, meanwhile shows that once a Fibonacci index is large enough, the only way to step to another Fibonacci number is if the scaling coefficient $F_{k+2} - F_{k-2}$ in \eqref{eq:F(k+1)F(m)} exactly equals $b^t$ and the remainder equals $F_{m-k}$; this rigidity arises because large Fibonacci numbers dominate the step expression, leaving no room for variation.

\begin{lemma}\label{lem:jump-bound}
Suppose that $m\geq 1$ and $F_{m+k}=b^t F_m + r$ with $1\leq t\leq N$ and $0\leq r < b^t$. Then
\[
F_{k+1} \;<\; 2\,b^N,
\]
and hence
\[
k \;\leq\; 1 + \log_\varphi(2b^N) \;=\; N\log_\varphi b \,+\, \log_\varphi 2 \,+\, 1.
\]
\end{lemma}

\begin{proof}
By hypothesis, we have
\begin{equation}
F_{m+k}=b^t F_m + r \;\leq\; b^N F_m + (b^N-1)\;<\; 2 b^N F_m \qquad(F_m\geq 1).
\label{eq:F(m+k)}
\end{equation}
Using \cref{lem:Miller2.1} and \eqref{eq:F(m+k)}, we see that $F_{k+1}F_m\leq F_{k+m}< 2 b^N F_m$, so $F_{k+1}< 2 b^N$, as desired. Using $F_{k+1}\geq \varphi^{(k+1)-2} = \varphi^{k-1}$ from \eqref{eq:phi-bounds} gives the claimed bound on $k$.
\end{proof}

\begin{lemma}\label{lem:rigidity}
We fix $b\geq 2$ and $N\geq 1$, and let
\[
m_* \;:=\; \Big\lceil \log_\varphi(2b^N) \Big\rceil + 4.
\]
For all $m\geq m_*$, if
\[
F_{m+k} \;=\; b^t F_m + r \qquad (1\leq t\leq N,\; 0\leq r<b^t, \; k\geq 1),
\]
then necessarily $k\geq 2$, $F_{k+2}-F_{k-2} = b^t$, $k$ is odd, and $r=F_{m-k}$.
\end{lemma}

\begin{proof}
First, we note that if $k=1$, then by hypothesis, $b^t F_m + r = F_{m+1} = F_m + F_{m-1}$, so
\begin{equation}
(b^t - 1)F_m \;=\; F_{m-1} - r.
\label{eq:k=1case}
\end{equation}
For $t\geq 1$, we have $b^t \geq 2$; hence, the left side of \eqref{eq:k=1case} is at least $F_m$, while the right side satisfies $F_{m-1} - r \leq F_{m-1} < F_m$. This is impossible, so we must have $k \neq 1$; thus, $k\geq 2$.

Now, for $k\geq2$, we set $\Delta := F_{k+2}-F_{k-2}-b^t$. We suppose for the sake of seeking a contradiction that $\Delta\neq 0$. Then, since $\Delta\in\mathbb{Z}$, we must have $\Delta\geq 1$ or $\Delta\leq -1$. 

By \cref{lem:jump-bound}, we must have $k \leq K := \lceil 1 + \log_\varphi(2b^N)\rceil$. Since $m\geq m_*=\lceil \log_\varphi(2b^N)\rceil+4 > K$ and $k\leq K$, we have $m>k$, so in particular $m\geq k\geq 2$ and \cref{lem:Miller2.2} applies. Thus, we have
\[
F_{m+k} \;=\; \big(F_{k+2}-F_{k-2}\big)F_m \,+\, (-1)^{k+1}F_{m-k};
\]
hence, our hypothesis that $F_{m+k} = b^t F_m + r$ implies that
\begin{equation}
r \;=\; \big(F_{k+2}-F_{k-2}-b^t\big)F_m \,+\, (-1)^{k+1}F_{m-k}\;=\; \Delta\,F_m \,+\, (-1)^{k+1} F_{m-k}.
\label{eq:r-isolation}
\end{equation}

Moreover, for $m\geq m_*$, \eqref{eq:phi-bounds} gives
\begin{equation}
F_{m-2} \;\geq\; \varphi^{(m-2)-2} \;\geq\; \varphi^{m-4} \;\geq\; \varphi^{\lceil \log_\varphi(2b^N)\rceil} \;\geq\; 2b^N \;>\; b^N \;\geq\; b^t.
\label{eq:F(m-2)-bound}
\end{equation}

And finally, as $m>k$, we have 
\begin{equation}
F_{m-k}> 0.
\label{eq:F(m-k)-bound}
\end{equation}

If $\Delta\geq 1$ then \eqref{eq:r-isolation} gives:
\[
\begin{cases}
\text{$k$ odd:}& r \;\ge\; F_m + F_{m-k} \;\ge\; b^t \qquad \text{(by \eqref{eq:F(m-2)-bound} and \eqref{eq:F(m-k)-bound})},\\[3pt]
\text{$k$ even:}& r \;\ge\; F_m - F_{m-k} \;\ge\; F_m - F_{m-1} \;=\; F_{m-2} \;>\; b^t \qquad \text{(by \eqref{eq:F(m-2)-bound})},
\end{cases}
\]
contradicting our hypothesis that $r<b^t$.

Meanwhile, if $\Delta\leq -1$ then \eqref{eq:r-isolation} gives:
\[
\begin{cases}
\text{$k$ odd:}& r \;\le\; -F_m + F_{m-k} \;\le\; -F_m + F_{m-1} \;=\; -F_{m-2} \;<\; 0,\\[3pt]
\text{$k$ even:}& r \;\le\; -F_m - F_{m-k} \;<\; 0,
\end{cases}
\]
contradicting our hypothesis that $r\geq 0$.

Thus, we see that we must have $\Delta=0$, i.e., $F_{k+2}-F_{k-2} = b^t$, and then $r = (-1)^{k+1} F_{m-k}$. Since $r\geq 0$ (by hypothesis) and $F_{m-k}>0$ (by \eqref{eq:F(m-k)-bound}), we must have $k$ odd, so $r=F_{m-k}$.
\end{proof}

With \cref{lem:jump-bound,lem:rigidity} in hand, we can now combine them to prove our main result.

\begin{proof}[Proof of \cref{thm:main}]
Let $m\geq n_*+K+1$. Suppose, for the sake of seeking a contradiction, that a step to another Fibonacci number appending at most $N$ base-$b$ digits is possible, i.e., there exist $k$, $t$ ($1\leq t\leq N)$, and $r$ ($0\leq r<b^t$) such that
\[
F_{m+k}=b^t F_m + r.
\]
By \cref{lem:jump-bound}, we know that $k\leq K$. Since $n_*\geq 2$ (as $b^N-1\geq 1=F_2$), we have
\[
m \;\geq\; n_* + K + 1 \;\geq\; K + 3 \;=\; \big\lceil 1 + \log_\varphi(2b^N)\big\rceil + 3 \;\geq\; \big\lceil \log_\varphi(2b^N)\big\rceil + 4 \;=\; m_*,
\]
so \cref{lem:rigidity} applies and yields $k\geq 2$, $F_{k+2}-F_{k-2}=b^t$, and $r=F_{m-k}$. Consequently, since $m\geq n_*+K+1$ and $k\leq K$, we have
\begin{equation}
r \;=\; F_{m-k} \;\geq\; F_{(n_*+K+1)-K} \;=\; F_{n_*+1}.
\label{eq:r-to-n*}
\end{equation}

By the definition of $n_*$ as the largest index $n$ such that $F_{n}\leq b^N - 1$, we have $F_{n_*+1} > b^N - 1$; hence, $F_{n_*+1} \geq b^N$. Thus \eqref{eq:r-to-n*} implies that $r \geq b^N > b^t - 1$, i.e., $r > b^t - 1$. Since $r < b^t$ and $r$ is an integer, we have reached a contradiction. Thus, no step to another Fibonacci number appending at most $N$ base-$b$ digits is possible once $m\geq n_*+K+1$.

Finally, to bound the walk length $L$ from an arbitrary starting index $m_0$: We first note that, by the preceding argument, if $m_0 \geq n_* + K + 1$, then $L=0$. Otherwise, if $m_0 < n_* + K + 1$, then each valid step increases the index by at least $1$ (as $k\geq 1$ in any step), so after at most $(n_* + K + 1 - m_0)$ steps we reach an index $\geq n_*+K+1$, at which no further step is possible. Hence $L \leq n_* + K$.

Using \eqref{eq:phi-bounds}, we have $n_* \leq 2 + N\log_\varphi b$ and $K \leq 2 + N\log_\varphi b + \log_\varphi 2$, which gives
\[
L \;\leq\; n_* + K \;\leq\; 4 + 2N\log_\varphi b + \log_\varphi 2 \;=\; 2N\log_\varphi b \,+\, \log_\varphi 2 \,+\, 4,
\]
as claimed.
\end{proof}

\begin{remark}[Comparison with the \citet{miller2022walking} Base-$10$ Bounds]
For $b=10$, \cref{thm:main} gives $L \leq 2N\log_\varphi 10 + O(1) \approx 9.57N + O(1)$, whereas~\cite{miller2022walking} gives an $O(\log N)$ bound with sharp constants by using the periodicity of the Fibonacci sequence modulo 10. Our method trades sharpness for universality: it applies to any base without periodicity computations.
\end{remark}

\section{Extension to Other Second-Order Linear Recurrences}

The proof strategy in \cref{thm:main} relies on three key properties of the Fibonacci sequence:
\begin{enumerate}[(i)]
	\item an \emph{addition formula} expressing $F_{m+k}$ in terms of $F_m$, $F_k$, and $F_{m-k}$ (\cref{lem:Miller2.2});
	\item a \emph{product bound} showing $F_{m+k}$ grows like $F_m \cdot F_k$ (\cref{lem:Miller2.1}); and
	\item \emph{exponential growth} with explicit constants (equation~\eqref{eq:phi-bounds}).
\end{enumerate}
These properties are not unique to Fibonacci numbers. Indeed, they hold for a much broader family of sequences---and as we show, all such sequences satisfy a version of our \cref{thm:main}.

We consider the family of sequences known as \emph{Lucas sequences}. For integer parameters $(P,Q)$, the \emph{Lucas sequence of the first kind $U_n = U_n(P,Q)$} and the companion \emph{Lucas sequence of the second kind $V_n = V_n(P,Q)$} are defined by
\begin{gather*}
U_0=0,\; U_1=1,\quad U_{n+1}=P\,U_n - Q\,U_{n-1},\\
V_0=2,\; V_1=P,\quad V_{n+1}=P\,V_n - Q\,V_{n-1},
\end{gather*}
respectively. The Fibonacci numbers arise in the case $(P,Q)=(1,-1)$, yielding $U_n=F_n$ and $V_n=L_n$, where $(L_n)$ denotes the classical Lucas number sequence.

Throughout this section, when we refer to ``the Lucas sequence'' $(U_n)$ without qualification, we mean the sequence of the first kind. (The sequence $(V_n)$ appears in our proofs through the addition formula relating $U_{m+k}$ to $U_m$, $V_k$, and $U_{m-k}$.)

\subsection{Lucas Sequence Properties}

When the characteristic polynomial $x^2 - Px + Q = 0$ has distinct real roots $\rho > 1$ and $\sigma = Q/\rho$, the sequence $(U_n)$ exhibits exponential growth---the condition $|Q|=1$ ensures that $|\sigma| = \rho^{-1} < 1$, so the smaller root's contribution decays exponentially. Moreover, all Lucas sequences satisfy versions of the addition and product formulas satisfied by the Fibonacci numbers.

\begin{lemma}\label{lem:lucas-properties}
Let $(U_n)$ be a Lucas sequence with integer parameters $(P,Q)$ satisfying $|Q|=1$ and $P^2 - 4Q > 0$. Let $\rho > 1$ be the larger root of $x^2 - Px + Q = 0$.  Then:
\begin{enumerate}[(a)]
	\item \label{pt:add} For all $m \geq k \geq 0$, we have
	\begin{equation}
	U_{m+k} \;=\; U_m V_k \;-\; Q^k\, U_{m-k};
	\label{eq:lucas-add}
	\end{equation}
	moreover, since $|Q|=1$, the coefficient $Q^k$ is bounded (indeed, $Q^k \in \{-1,+1\}$).

	\item \label{pt:prod} There exist constants $A, B > 0$ depending only on $(P,Q)$ such that
	\[
	A\, U_{k+1}\,U_m \;\leq\; U_{m+k} \;\leq\; B\, U_{k+2}\, U_m\qquad(m\geq k\geq 1).
	\]

	\item \label{pt:exp} There exist constants $c_1, c_2 > 0$ depending only on $(P,Q)$ such that
	\[
	c_1 \rho^{n-1} \;\leq\; U_n \;\leq\; c_2 \rho^{n-1}\qquad(n\geq 1).
	\]
\end{enumerate}
\end{lemma}

\begin{proof}
Part~\ref{pt:add} is a standard identity in the theory of Lucas sequences (see, e.g., \cite[Equation~(2.8)]{ribenboim2000my}). Part~\ref{pt:exp} follows from Binet-type formulas for Lucas sequences (see \cite[Theorem~10.1]{koshy2019fibonacci} or \cite[Equation~(2.1)]{ribenboim2000my}).

For Part~\ref{pt:prod}, we use Part~\ref{pt:exp}: Since $U_{m+k} \geq c_1 \rho^{m+k-1}$ and $U_m U_{k+1} \leq c_2^2 \rho^{m+k-1}$, we have
\[
U_{m+k} \;\geq\; \frac{c_1}{c_2^2} \, U_m U_{k+1},
\]
giving the lower bound with $A = \frac{c_1}{c_2^2}$.  Similarly, since $U_{m+k} \leq c_2 \rho^{m+k-1}$ and $U_m U_{k+2} \geq c_1^2 \rho^{m+k}$, we have
\[
U_{m+k} \;\leq\; \frac{c_2}{c_1^2 \rho} \, U_m U_{k+1},
\]
proving the upper bound with $B = \frac{c_2}{c_1^2 \rho}$.
\end{proof}

\subsection{Walking Along a Lucas Sequence}

As the intuition given at the start of this section suggests, the three properties established in \cref{lem:lucas-properties} imply a version of \cref{thm:main}. 

\begin{theorem}\label{thm:lucas}
Fix $b\geq 2$ and $N\geq 1$, and let $(U_n)$ be a Lucas sequence with integer parameters $(P,Q)$ satisfying:
\begin{itemize}
\item $|Q|=1$ (ensuring the tail term in the addition formula is bounded);
\item $P^2 - 4Q > 0$ (ensuring exponential growth); and
\item $U_n > 0$ for all sufficiently large $n$ (ensuring positivity).
\end{itemize}

Then every at-most-$N$-digit-appending walk on $(U_n)$ terminates after at most $L$ steps, where
\[
L \;\leq\; C_1\, N \log_\rho b \,+\, C_2
\]
for constants $C_1,C_2>0$ depending only on $(P,Q)$.
\end{theorem}

\begin{proof}[Proof Sketch]
The proof follows the same strategy as in our proof of \cref{thm:main}, replacing Fibonacci-specific identities with their Lucas sequence analogues from \cref{lem:lucas-properties}:

\begin{description}
\item[Index jump bound] Using the product comparability from Part~\ref{pt:prod} of~\cref{lem:lucas-properties}, we see that if $U_{m+k} = b^t U_m + r$ with $r < b^t \leq b^N$, then $U_{m+k} < 2b^N U_m$. This gives $A U_{k+1} U_m \leq 2b^N U_m$; hence, $U_{k+1} < 2b^N/A$. Combined with the exponential growth bound from Part~\ref{pt:exp} of~\cref{lem:lucas-properties}, we obtain
\[
k \;\leq\; C'_1 N\log_\rho b \,+\, C'_2=:K^{(U_n)}
\]
for constants $C'_1, C'_2$ depending on $(P,Q)$.

\item[Rigidity] We choose a threshold $m_*^{(U_n)}$ so that, for all $m \geq m_*^{(U_n)}$: (i) $U_{m-2} > b^N$; (ii) $U_m>0$ and $U_n$ is strictly increasing for all $n \geq m$; and (iii) if $Q=1$, then also $U_m - U_{m-1} > b^N$. (Such an $m_*^{(U_n)}$ exists by the exponential growth in Part~\ref{pt:exp} of~\cref{lem:lucas-properties}, combined with our positivity hypothesis.) Now, we suppose that for some $m \ge m_*^{(U_n)}$ there is a valid step $U_{m+k} = b^t U_m + r$ with $1 \leq t \le N$ and $0 \leq r < b^t$. By the index-jump bound, $k$ is bounded independently of $m$; in particular, for $m \ge m_*^{(U_n)}$ we have $m>k$. Using the addition formula (Part~\ref{pt:add} of~\cref{lem:lucas-properties}), we then obtain
\[
r \;=\; \big(V_k - b^t\big) U_m \;-\; Q^k\, U_{m-k}.
\]
Taking $\Delta := V_k - b^t \in \mathbb{Z}$, for $m \ge m_*^{(U_n)}$, the main term $\Delta U_m$ dominates while $U_{m-k} \le U_{m-1}$; a brief sign/size check as in the proof of~\cref{lem:rigidity} shows that $\Delta \geq 1$ would force $r \geq b^t$ and $\Delta \leq -1$ would force $r < 0$---both impossible. Hence, like in the proof of~\cref{lem:rigidity}, we must have $\Delta=0$, which implies rigid structure on the specific Lucas sequence member reached by a valid step: $V_k = b^t$, and therefore $r = -Q^k U_{m-k}$. Then, because $U_{m-k}>0$ and $r \ge 0$, we conclude that if $Q=-1$, then $k$ must be odd and $r = U_{m-k}$ (the exact Fibonacci analogue); if $Q=1$, then $r = -U_{m-k} < 0$, which is impossible---so no step can occur for $m \ge m_*^{(U_n)}$ in the case $Q=1$. 

\item[Termination] We define $n_*^{(U_n)} := \max\{n\geq 0: U_n \leq b^N-1\}$. For $m \geq n_*^{(U_n)} + K^{(U_n)} + 1$, rigidity forces the remainder to exceed $b^N$, contradicting the digit-appending constraint. As each step increases the index by at least $1$, the walk terminates in at most $n_*^{(U_n)} + K^{(U_n)} + 1= O_{P,Q}(N\log_\rho b)$ steps.\qedhere
\end{description}
\end{proof}

\subsection{Illustration and Discussion}

The Fibonacci sequence satisfies all three conditions of \cref{thm:lucas} with $(P,Q) = (1,-1)$, so our \cref{thm:main} is a special case of the result (with explicitly derived constants). Meanwhile, the \emph{Pell numbers} $P_n$ (defined by $P_0 = 0$, $P_1 = 1$, and $P_{n+1} = 2P_n + P_{n-1}$) form a Lucas sequence with $(P,Q) = (2,-1)$, satisfying $|Q|=1$ and growing like $\rho^n$ where $\rho = 1 + \sqrt{2} \approx 2.414$. Therefore, \cref{thm:lucas} guarantees that any at-most-$N$-digit-appending walk on the Pell sequence terminates in $O(N\log b)$ steps.

It also seems likely that analogous identities and rigidity phenomena may hold for higher-order linear recurrences, although we do not pursue that here. The key obstacle is that for order $\ell>2$, the analogue of \eqref{eq:lucas-add} would involve $\ell$ terms, so the rigidity step would need to match $\ell-1$ coefficients simultaneously, which would require more delicate analysis.

\end{document}